\documentclass[reqno, 11pt]{amsart}
\usepackage{amsthm,amsmath,amssymb,graphicx,tikz,
  enumerate,paralist,bbm,mathrsfs}
\usepackage[margin=1in]{geometry}

\usetikzlibrary{decorations.pathreplacing}

\usepackage[color,final]{showkeys} %add in 'final' into parameter to remove
\usepackage[pdfstartview=FitV,hidelinks]{hyperref}

\colorlet{refkey}{orange!20}
\colorlet{labelkey}{blue!30}
\newtheorem{theorem}{Theorem}[section]
\newtheorem{proposition}[theorem]{Proposition}
\newtheorem{lemma}[theorem]{Lemma}
\newtheorem{corollary}[theorem]{Corollary}
\newtheorem{conjecture}[theorem]{Conjecture}

\theoremstyle{definition}
\newtheorem{definition}[theorem]{Definition}
\newtheorem{example}[theorem]{Example}

\theoremstyle{remark}

\newtheorem*{notation}{Notation}

\newcommand{\abs}[1]{\left\lvert#1\right\rvert}
\newcommand{\norm}[1]{\left\lVert#1\right\rVert}
\newcommand{\snorm}[1]{\lVert#1\rVert}
\newcommand{\ang}[1]{\left\langle #1 \right\rangle}

\newcommand{\paren}[1]{\left( #1 \right)}

\newcommand{\set}[1]{\left\{ #1 \right\}}

\newcommand{\wt}{\widetilde}

\renewcommand{\epsilon}{\varepsilon}
\newcommand{\del}{\partial}
\newcommand{\dashto}{\dashrightarrow}

\newcommand{\x}{\times}

\newcommand{\e}{\epsilon}

\newcommand{\GG}{\mathbb{G}}
\newcommand{\RR}{\mathbb{R}}

\newcommand{\NN}{\mathbb{N}}

\newcommand{\cA}{\mathcal{A}}
\newcommand{\cB}{\mathcal{B}}

\newcommand{\cW}{\mathcal{W}}

\newcommand{\cP}{\mathcal{P}}
\newcommand{\cQ}{\mathcal{Q}}
\newcommand{\cR}{\mathcal{R}}

\newcommand{\bx}{\mathbf{x}}
\newcommand{\bU}{\mathbf{U}}
\newcommand{\bW}{\mathbf{W}}

\newcommand{\sP}{\mathscr{P}}
\newcommand{\sQ}{\mathscr{Q}}

\title{Hypergraph limits: a regularity approach}
\author{Yufei Zhao}
\address{Department of Mathematics\\ MIT\\ Cambridge, MA 02139-4307.}
\email{yufeiz@math.mit.edu}

\begin{document}

\begin{abstract}
  A sequence of $k$-uniform hypergraphs $H_1, H_2, \dots$ is
  convergent if the sequence of homomorphism densities $t(F, H_1),
  t(F, H_2), \dots$ converges for every $k$-uniform hypergraph
  $F$. For graphs, Lov\'asz and Szegedy showed that every convergent
  sequence has a limit in the form of a symmetric measurable function
  $W \colon [0,1]^2 \to [0,1]$. For hypergraphs, analogous limits $W
  \colon [0,1]^{2^k-2} \to [0,1]$ were constructed by Elek and Szegedy
  using ultraproducts. These limits had also been studied earlier by
  Hoover, Aldous, and Kallenberg in the setting of exchangeable random
  arrays.

  In this paper, we give a new proof and construction of hypergraph
  limits. Our approach is inspired by the original approach of
  Lov\'asz and Szegedy, with the key ingredient being a weak
  Frieze-Kannan type regularity lemma.
\end{abstract}

\maketitle

\section{Introduction} \label{sec:intro}

One of the starting points in the theory of dense graph limits is the
seminal paper by Lov\'asz and Szegedy~\cite{LS06} where they
constructed limit objects for convergent sequences of dense
graphs. The subject has grown enormously since then with many exciting
developments (see Lov\'asz's recent monograph~\cite{L12}).

For any two graphs $F$ and $G$, let $\hom(F,G)$ denote the number of
homomorphism from $F$ to $G$, i.e., maps $V(F) \to V(G)$ that carry
every edge of $F$ to an edge of $G$. The homomorphism density $t(F,G)$
is
defined to be the probability that a random map $V(F) \to V(G)$ is a
homomorphism, i.e.,
\[
t(F,G) := \frac{\hom(F,G)}{\abs{V(G)}^{\abs{V(F)}}}.
\]
A sequence of graphs $G_1, G_2, \dots$ is called \emph{convergent} if
the sequence $t(F,G_1), t(F, G_2), \dots$
converges for every graph $F$. Convergent graph sequences were defined and
studied in~\cite{BCLSV08,BCLSV12}. The main result of Lov\'asz and
Szegedy~\cite{LS06} is that for every
convergent graph sequence there is a limit object in the form of a
\emph{graphon}, which is a symmetric measurable function $W \colon [0,1]^2
\to
[0,1]$ (here
\emph{symmetric} means that $W(x,y) = W(y,x)$) such that $t(F,G_n) \to
t(F,W)$ as $n \to \infty$ for all graphs $F$. Here $t(F,W)$ is defined by
\[
t(F,W) := \int_{[0,1]^{V(F)}} \prod_{ij \in E(F)} W(x_i,x_j) \,dx_1dx_2 \cdots dx_{\abs{V(F)}}
\]
The natural extension of these limits to hypergraphs was considered by
Elek and
Szegedy~\cite{ES12}. They constructed using ultraproducts an
``ultralimit hypergraph'' for any sequence of
hypergraphs, and established a
correspondence principle which enabled them to convert statements about
finite hypergraphs, such as hypergraph regularity and removal
lemmas~\cite{G07,NRS06,T06}, to measure-theoretic claims about ultralimit
spaces. One of the
consequences of their work is the existence of a limit object in the form
of
a measurable functions $W \colon [0,1]^{2^k-2} \to [0,1]$ for any
convergent sequence of $k$-uniform hypergraphs.

These limit objects had actually appeared earlier in a different form, in
the
study of exchangeable random arrays, initiated by Hoover~\cite{Hoo79},
Aldous~\cite{Ald81}, and Kallenberg~\cite{Kal92} during the 1980s,
building on
the classic de Finnetti's theorem on exchangeable random variables. This
connection is explained in the survey~\cite{Aus08} by Austin, where he
credits
Tao~\cite{Tao07} for initiating the link between exchangeable random
variables
and hypergraphs. These connections for graphs are also explained in the
survey
by Diaconis and Janson~\cite{DJ08} as well as Aldous' ICM
talk~\cite{Ald10}.

The purpose of this paper is to provide a new proof of the existence of
hypergraph limits. Our approach is based on weak Frieze-Kannan~\cite{FK99}
type
regularity partitions, in line with mainstream perspectives on dense graph
limits.
The proof does not use any exchangeable random variables or ultraproducts,
and
the construction of the limit is subjectively more concrete than earlier
proofs.
Our proof is inspired by the original approach of Lov\'asz and
Szegedy~\cite{LS06}, and the paper is self-contained other than an
application
of the
Martingale Convergence Theorem.

\subsection{Convergence and limit object}
For any $k$-uniform hypergraphs $F$ and $H$, let $\hom(F,H)$
denote the number of homomorphisms from $F$ to $H$, i.e., maps $V(F)
\to V(H)$ that carry every edge of $F$ to an edge of $H$. Define
$t(F,H) := \hom(F,H) / \abs{V(H)}^{\abs{V(F)}}$. This is the
probability that a random map $V(F) \to V(H)$ is a homomorphism.

\begin{definition}[Convergence]
  A
  sequence of $k$-uniform hypergraphs $H_1, H_2, \dots$ is called
  \emph{convergent}
  if the sequence $t(F,H_1), t(F,H_2), \dots$ converges for every
  $k$-uniform
  hypergraph $F$.
\end{definition}

For any positive integer $n$, define $[n]:=\set{1,2,\dots,n}$. For any set
$A$,
define $r(A)$ to be the
collection of all nonempty subsets of $A$, and $r_<(A)$ to be
collection of all nonempty proper subsets of $A$. More generally, let
$r(A,m)$ denote
the collection of all nonempty subsets of $A$ of size at most $m$. So
for instance, $r_<([k]) = r([k],k-1)$. We will also use the shorthand
$r[k]$ and
$r_<[k]$ to mean $r([k])$ and $r_<([k])$ respectively.

Any permutation $\sigma$ of a set $A$ induces a permutation
on $r(A,m)$. We say that a function $W \colon [0,1]^{r([k],m)} \to
[0,1]$ is \emph{symmetric} if it remains invariant under any
permutation of the coordinates induced by any permutation of
$[k]$. For example, $W \colon [0,1]^{r_<[3]} \to [0,1]$ being symmetric
means that
\begin{equation}\label{eq:3-symm}
  W(x_1,x_2,x_3,x_{12},x_{13},x_{23}) =
  W(x_{\sigma_1},x_{\sigma_2},x_{\sigma_3},x_{\sigma_1\sigma_2},x_{\sigma_1\sigma_3},x_{\sigma_2\sigma_3})
\end{equation}
for any permutation $\sigma$ of $\{1,2,3\}$. Here we write $x_i$ for
$x_{\set{i}}$ and $x_{ij}$ for $x_{\set{i,j}}$.

\begin{definition}
  A \emph{$k$-uniform hypergraphon} is a symmetric measurable function
  $W \colon [0,1]^{r_<([k])} \to [0,1]$.
\end{definition}

\begin{example}
  A 3-uniform hypergraphon is a measurable function $W \colon [0,1]^6
  \to [0,1]$ satisfying the symmetry condition \eqref{eq:3-symm}.
\end{example}

For
any $k$-uniform hypergraph $F$ and hypergraphon $W$, define
the homomorphism density by
\[
t(F,W) := \int_{[0,1]^{r(V(F),k-1)}} \prod_{A \in E(F)} W(\bx_{r_<(A)})
\,d\bx
\]
Our convention throughout the paper is that if
$\bx = (x_A: A \in \cA) \in [0,1]^{\cA}$ is a vector whose coordinates are
indexed by
some set system $\cA$, and $\cB \subseteq \cA$ is a subcollection, then we
write
$\bx_\cB = (x_B : B \in \cB) \in [0,1]^{\cB}$ to mean the restriction of
the
vector
to the coordinates indexed by $\cB$.

\begin{example}
  If $K_{4}^{(3)} = \{123,124,134,234\}$ is the complete 3-uniform
  hypergraph on
  4
  vertices and $W$ is a 3-uniform hypergraphon, then
  \begin{multline*}
    t(K_4^{(3)},W)
    = \int_{[0,1]^{10}} W(x_1,x_2,x_3,x_{12},x_{13},x_{23})
    W(x_1,x_2,x_4,x_{12},x_{14},x_{24})
    W(x_1,x_3,x_4,x_{13},x_{14},x_{34}) \cdot \\
    \cdot
    W(x_2,x_3,x_4,x_{23},x_{24},x_{34})
    \,dx_1dx_2dx_3dx_4dx_{12}dx_{13}dx_{14}dx_{23}dx_{24}dx_{34}.
  \end{multline*}
\end{example}

Every $k$-uniform hypergraph $H$ can be represented as a $k$-uniform
hypergraphon $W^H$ as follows: divide $[0,1]$ into $\abs{V(H)}$
equal-length intervals $\set{I_1, I_2, \dots, I_{\abs{V(H)}}}$. For each
$\bx
\in [0,1]^{r_<[k]}$ define
\[
W^H(\bx)
:= \begin{cases}
  1 & \text{if } x_{\set{i}} \in I_{a_i} \text{ for } i = 1, \dots, k
  \text{ and } \set{a_1, a_2, \dots, a_k} \text{ is an edge of $H$},
  \\
  0 & \text{otherwise.}
\end{cases}
\]
In particular, $W^H(\bx)$ depends only on the $k$ coordinates of
$\bx$ corresponding to subsets of $[k]$ of size 1. It can be
alternatively described as transforming the adjacency array of $H$ into a
$\{0,1\}$-valued step function and
then adding $2^k-2-k$ extra free coordinates. Observe
that $t(F,H) = t(F,W^H)$ for every $k$-uniform hypergraph $F$.

The main purpose of this paper is to give a new proof of the following
result \cite[Thm.~7]{ES12} on the existence of
hypergraph limits.

\begin{theorem}
  \label{thm:hypergraph-converge-limit}
  If $H_1, H_2, \dots$ is a convergent sequence of $k$-uniform
  hypergraphs, then there exists a $k$-uniform hypergraphon $W$ so
  that $t(F,H_n) \to t(F,W)$ as $n \to \infty$ for every $k$-uniform
  hypergraph $F$.
\end{theorem}

\subsection{Why are there so many coordinates?} \label{sec:many-coord}
It may initially seem somewhat strange that we need 6 coordinates
to describe the limit of 3-uniform hypergraphs, whereas every
3-uniform hypergraph can be described in terms of a 3-dimensional
adjacency array. These extra dimensions do not arise for limits of
graphs, but they are essential for hypergraphs. Here is a
standard example illustrating why functions of the form $[0,1]^3
\to [0,1]$ cannot capture the richness of 3-uniform hypergraph
limits. Take $G_n \sim \GG(n,1/2)$ to be a sequence of graphs on $n$
vertices, where each edge is generated with probability $1/2$, and let
$H_n$ be the 3-uniform hypergraph whose edges are the triangles of
$G_n$. Then with probability one,
$t(F, H_n) \to 2^{-|\del F|}$ for every 3-uniform hypergraph $F$, where
$\del F$
is the collection of unordered pairs of
vertices of $F$ that are contained in some edge of $F$. The limit of
$H_n$ is different from, say, the constant hypergraphon $1/2$, which
is the limit of a sequence of 3-uniform hypergraphs where every triple
of vertices is taken to be an edge independently with probability
$1/2$. To describe the limit of $H_n$, we need to incorporate
the limit of $G_n$ into the data, and this is achieved by the three
extra coordinates. We know that the graph sequence $G_n$ converges to the
constant graphon with value $1/2$. To build the limit of $H_n$, we
partition
each of the last three coordinates, $x_{12},x_{13},x_{23}$ into two
intervals
$[0,1/2]\cup(1/2,1]$, corresponding to the limit of $G_n$ and the
limit of its
complement.
The limiting hypergraphon has constant value 1 on $[0,1]^3 \x [0,1/2]^3$
(as the
edges of $H_n$ are supported on $G_n$) and 0
elsewhere. Intuitively, the first three coordinates encode the
vertex types, the last three coordinates encode the vertex-pair
types. This hypergraphon is $\set{0,1}$-valued since it is
deterministic once the vertex and vertex-pairs types are set. If we
modify the sequence $H_n$ so that each triangle of $G_n$ is included
as an edge of $H_n$ with some probability $p$ independently, then the
limiting hypergraphon would be constant $p$ on $[0,1]^3 \x
[0,1/2]^3$ and 0 elsewhere.

For $k$-uniform hypergraphs, we can similarly impose some
structure at each level, corresponding to $j$-element subsets of vertices,
for
every $1 \leq j \leq k$. This is why we need a coordinate for every proper
subset of $[k]$ to describe hypergraph limits.

\subsection{Random hypergraph model}

To further illustrate the involvement of the $2^k-2$ coordinates in a
hypergraphon, let us review the associated random hypergraph model.

Recall that if $W \colon [0,1]^2 \to [0,1]$ is a graphon, then we have
the following natural random graph model $\GG(n,W)$ on $n$ vertices:
choose i.i.d.\ uniform $x_1, x_2, \dots, x_n \in [0,1]$, and let there
be an edge between vertices $i$ and $j$ with probability $W(x_i,x_j)$
independently. It was shown \cite[Cor.~2.6]{LS06} using Azuma's
inequality that $\GG(n,W)$
converges to the limit $W$ almost surely.

Similarly, a $k$-uniform hypergraphon $W$ gives a natural model $\GG(n,W)$
of a
random $k$-uniform hypergraph on $n$ vertices: choose a uniformly
random $\bx \in [0,1]^{r([n],k-1)}$ and add the edge $B =
\set{i_1, \dots, i_k} \subseteq [n]$ with probability
$W(\bx_{r_<(B)})$ independently. Essentially the same proof for graphs
extend over to show \cite[Thm.~11]{ES12} that
$\GG(n,W)$ converges to $W$ in the sense of
Theorem~\ref{thm:hypergraph-converge-limit}, as $n \to
\infty$ with probability one. Observe that the random hypergraphs $H_n$ of
triangles in $\GG(n,1/2)$ discussed earlier is a special case of this
model.

\subsection{Analytic version and compactness}
It will be convenient to prove an analytic version of
Theorem~\ref{thm:hypergraph-converge-limit}. We say
that a sequence of $k$-uniform hypergraphons $W_1, W_2, \dots$ is
\emph{convergent} if the sequence $t(F, W_1), t(F, W_2), \dots$ converges
for
every
$k$-uniform hypergraph $F$.

\begin{theorem}
  \label{thm:hypergraphon-converge-limit}
  If $W_1, W_2, \dots$ is a convergent sequence of $k$-uniform
  hypergraphons,
  then there exists a $k$-uniform hypergraphon $\wt W$ so
  that $t(F,W_n) \to t(F,\wt W)$ as $n \to \infty$ for every $k$-uniform
  hypergraph $F$.
\end{theorem}

In this case we say that $W_n$ converges to $\wt W$. Here is an equivalent
formulation of the theorem.

\begin{theorem} \label{thm:hypergraphon-subseq}
  Every sequence $W_1, W_2, \dots$ of $k$-uniform hypergraphons contains
  a
  subsequence that converges to some $k$-uniform hypergraphon $\wt W$.
\end{theorem}

Theorem~\ref{thm:hypergraphon-subseq} implies
Theorem~\ref{thm:hypergraphon-converge-limit}  trivially since we can just
take
the limit $\wt W$ produced by Theorem~\ref{thm:hypergraphon-subseq}. The
converse is true because $[0,1]^\NN$ is sequentially compact, so we can
restrict
$(W_n)$ to some subsequence $(W_{n_i})$ so that $t(F, W_{n_i})$ converges
as
$i\to \infty$ for every $F$.

We shall prove Theorem~\ref{thm:hypergraphon-subseq} with respect
to
another
notion of convergence based on regular partitions, which implies the
convergence of
homomorphism densities. The partition-based
convergence gives some structural insight into the convergence of
hypergraphs.

There is a neat interpretation of
Theorem~\ref{thm:hypergraphon-subseq} in terms of
compactness, discovered by Lov\'asz and Szegedy~\cite{LS07} in the case of
graphons. Let $\cW^{(k)}_0$ denote the set of $k$-uniform
hypergraphons. Give $\cW^{(k)}_0$ the weakest topology for which
the functions $t(F,\cdot)$ are continuous for every $k$-uniform
hypergraph $F$. Identify
$W$ with $W'$ if $t(F,W) = t(F,W')$ for every
$k$-uniform hypergraph $F$. Call this topology the
\emph{left-convergence topology} of $\cW^{(k)}_0$.

\begin{corollary}
  \label{cor:compact}
  The space $\cW_0^{(k)}$ with the
  left-convergence topology is compact.
\end{corollary}

\begin{proof}
  The space is metrizable with the metric $\delta(W, W') =
  \sum_{i \geq 1} 2^{-i} \abs{t(F_i,W) - t(F_i,W')}$ where
  $(F_i)$ is some enumeration of all isomorphism classes of
  $k$-uniform hypergraphs. We know that compactness is equivalent to
  sequential compactness in metric spaces, and
  Theorem~\ref{thm:hypergraphon-subseq} shows that the space is
  sequentially
  compact.
\end{proof}

When $k=2$, Lov\'asz and Szegedy~\cite{LS07} showed that $\cW_0^{(2)}$
is compact under the cut metric topology, and Borgs, Chayes, Lov\'asz,
S\'os, and Vesztergombi~\cite{BCLSV08} showed that the cut metric topology
is
equivalent to the left-convergence topology. Lov\'asz and Szegedy
interpreted the compactness with respect to the cut metric as an
analytic form of the regularity lemma, and they showed that the
compactness of the space of graphons
implies strong versions of the regularity lemma. Unfortunately,
for $k \geq 3$,
we do not know of a useful extension of the cut metric to hypergraphs
(and there may be some reasons to believe that such a natural metric might
be too much to ask for). This is one of the main obstacles in working with
convergence of hypergraphs. It would be nice to have a simple and useful
description of distance between hypergraphs which agrees with the topology
induced by
homomorphism densities.

\subsection{Organization} In \S\ref{sec:graphs} we review the
Lov\'asz-Szegedy
construction of graph limits. In \S\ref{sec:3-limit} we give an
informal sketch of the proof of the existence of 3-uniform hypergraph
limits.
Most of the ideas, minus the technical hairiness, are contained in
\S\ref{sec:3-limit}. The proof of the main result is contained in
\S\ref{sec:notation}--\ref{sec:branch}. \S\ref{sec:notation} collects some
of
the notation used in the proof. \S\ref{sec:reg-count} contains the
regularity
and counting lemmas central to the proof. In \S\ref{sec:branch} we
introduce
branching partitions and formulate the notion of partitionable convergence,
which implies, via counting lemmas, the convergence of homomorphism
densities.
We then prove the existence of limits with respect to partitionable
convergence.

\section{Limits of graphons} \label{sec:graphs}

For any symmetric measurable function $W \colon [0,1]^2 \to \RR$, the
cut norm is defined by
\begin{equation} \label{eq:cut}
  \norm{W}_\square := \sup_{S,T \subseteq [0,1]} \abs{\int_{S\x T} W(x,y)
    \,dxdy},
\end{equation}
where $S$ and $T$ range over all measurable subsets of $[0,1]$. We have
the identity
\begin{equation}
  \label{eq:cut-functional}
  \norm{W}_\square = \sup_{u,v \colon [0,1] \to [0,1]} \abs{\int
    W(x,y)u(x) v(y) \,dxdy}
\end{equation}
where $u$ and $v$ range over all measurable functions $[0,1] \to
[0,1]$. Indeed, since the integral in \eqref{eq:cut-functional} is linear
in
both $u$ and $v$, one can restrict to $\{0,1\}$-valued $u$ and $v$, thereby
reducing to \eqref{eq:cut}.

Recall that a graphon is a symmetric measurable function $W \colon
[0,1]^2 \to [0,1]$.
For any measure preserving bijection $\phi \colon [0,1] \to [0,1]$ and
any graphon $W$, define $W^\phi$ by $W^\phi (x,y) =
W(\phi(x),\phi(y))$. We define the cut distance between graphons by
\[
\delta_\square (U, W) = \inf_\phi \snorm{U^\phi - W}_\square,
\]
where the infimum is taken over all measure preserving bijections
$\phi \colon [0,1] \to [0,1]$. The cut distance can be defined for
pairs of graphs by considering their associated graphons. Graphs that are close in cut
distance are also close in homomorphism densities, by the following
counting
lemma.

\begin{lemma}[Counting lemma]
  \label{lem:2-count}
  For any graphons $U$ and $W$ and any graph $F$, we have
  \[
  \abs{t(F, U) - t(F,W)} \leq e(F) \norm{U - W}_\square
  \]
  where $e(F)$ is the number of edges of $F$.
\end{lemma}

We illustrate the proof through the example $F = K_3$.
\begin{align*}
  &t(K_3,U) - t(K_3,W)
  \\
  &= \int_{[0,1]^3} (U(x,y)U(x,z)U(y,z) - W(x,y)W(x,z)W(y,z)) \,dxdydz
  \\
  &= \int_{[0,1]^3} (U(x,y)-W(x,y))W(x,z)W(y,z) \,dxdydz \\
  &\qquad    +  \int_{[0,1]^3} U(x,y)(U(x,z)-W(x,z))W(y,z) \,dxdydz \\
  &\qquad    +  \int_{[0,1]^3} U(x,y)U(x,z)(U(y,z)-W(y,z)) \,dxdydz
\end{align*}
Each of the three terms in the final sum is bounded in absolute value by
$\norm{U -
  W}_\square$. For example, for the first term, for every fixed value of
$z$,
the integral has the form \eqref{eq:cut-functional}, and so it is bounded in
absolute value by $\norm{U - W}_\square$, and the same bound holds after
integrating $z$ by the triangle inequality. It follows that $|t(K_3,U) -
t(K_3,W)| \leq 3\norm{U - W}_\square$.

\medskip

For any graphon $W$ and any partition $\cQ$ of $[0,1]$ into a finite
collection
of measurable
subsets, let $W_\cQ$ be a graphon which is
the step function obtained from $W$ by replacing its
value at $(x,y) \in Q_i \x Q_j$ by the
average value of $W$ on $Q_i \x Q_j$, for any $Q_i, Q_j \in \cQ$, (if
either
$Q_i$
or
$Q_j$ has measure zero, then assign value 0 on $Q_i \x
Q_j$). For graphs, think of $\cQ$ as a partition of the vertex set, and
$W_\cQ$
as recording the edge densities between pairs of vertex subsets.

A key tool in the construction of graph limits is the following
weak regularity lemma due to Frieze and
Kannan~\cite{FK99} (see also \cite[Lem~3.1]{LS07}). It can be proved by an
$L^2$-energy increment argument.

\begin{lemma}[Weak regularity lemma]
  \label{lem:2-weak-reg}
  For every $\e > 0$ and every symmetric measurable function $W \colon
  [0,1]^2
  \to
  [0,1]$, there is some partition $\cQ$ of $[0,1]$ into at most
  $2^{2/\e^2}$ parts such that $\norm{W - W_\cQ}_\square \leq \e$.
\end{lemma}

Lov\'asz and Szegedy~\cite{LS07} showed that with respect to the cut
metric,
after
identifying graphons with cut distance zero, the space of all graphons
is compact. Equivalently:

\begin{theorem}[Lov\'asz and Szegedy~\cite{LS07}]
  \label{thm:LS-analytic}
  Every sequence $W_1, W_2, \dots$ of graphons contains a subsequence
  converging to some graphon $\wt W$ in cut distance.
\end{theorem}

Let us recall the idea of the proof of Theorem~\ref{thm:LS-analytic}.
Let $\e > 0$. We apply the weak regularity lemma to approximate
every $W_n$ by some $(W_n)_{\cQ_n}$. By
replacing each $W_n$ by some $W_n^{\phi_n}$ for some measure preserving
bijection $\phi_n$, we may assume that the
partition $\cQ_n$ divides $[0,1]$ into intervals. Take a
subsequence so that the lengths of the intervals converge, and the
values of $(W_n)_{\cQ_n}$ inside the boxes induced by the partition also
converge, i.e., the value inside the $(i,j)$-th box of $(W_n)_{\cQ_n}$
converges to some value as $n \to \infty$ (may be different limits for
different $(i,j)$). Then in this subsequence, $(W_n)_{\cQ_n}$ converges
pointwise almost everywhere to some limit $\wt U_1$, which is also a
step function.

Now repeat the same procedure with a smaller $\e' < \e$. We obtain new
partitions $\cQ'_n$ which are refinements of previous partitions. Call
the resulting limit $\wt U_2$. Note that steps of $(W_n)_{\cQ'_n}$ are
refinements of the steps of $(W_n)_{\cQ_n}$, and the values of the
latter can be obtained from the former by averaging over each
step. Thus a similar relation holds for $\wt U_2$ and $\wt U_1$.

Now we repeat this procedure for a sequence of $\e_k$ tending to
zero. We obtain a sequence $\wt U_1, \wt U_2, \dots$ of step functions so
that
each $\wt U_s$ can be obtained from $\wt U_{s+1}$ by average over each
step. It follows that if $(X,Y)$ is a uniform random point in
$[0,1]^2$, then the sequence $(\wt U_1(X,Y), \wt U_2(X,Y), \dots)$ is a
martingale. Since every $\wt U_s$ is bounded, the Martingale Convergence
Theorem\footnote{The Martingale Convergence Theorem
  (see~\cite[Thm.~11.5]{W91})
  says that every $L^1$-bounded martingale converges almost surely. Our
  martingales
  are actually bounded uniformly within $[0,1]$.}
implies that the martingale
converges with
probability 1, and hence there is some $\wt W \colon [0,1]^2 \to [0,1]$
which is the pointwise almost everywhere limit of $\wt U_s$'s. One
then checks that $\wt W$ is the desired limit.

In summary, the above proof consists of two main steps:
\begin{enumerate}
\item For each error tolerance $\e$, apply a weak regularity lemma to
  get a finite-dimensional step function approximation of each graphon.
  Take a
  subsequence so that the step functions converge.
\item Take a decreasing sequence of $\e$ tending to zero, we obtain
  refining
  chains of regularity partitions, and the corresponding subsequential
  limits $\wt
  U_s$ form a
  martingale. The existence of the final limit graphon follows by the
  Martingale
  Convergence Theorem.
\end{enumerate}

\section{Limits of 3-uniform hypergraphs} \label{sec:3-limit}

In this section we sketch the idea for 3-uniform hypergraph limits.
To keep things simple, consider a sequence $H_1, H_2, \dots$ of 3-uniform
hypergraphs (as opposed to hypergraphons).

We begin with an initial attempt that does not quite work. For a 3-variable
function $W \colon [0,1]^3 \to \RR$, we might extend the cut norm
\eqref{eq:2-cut} as follows (assume everything is measurable from now
on):
\begin{equation} \label{eq:3-vertex-cut}
  \text{(bad cut norm)}\quad
  \norm{W}_\square = \sup_{R,S,T \subseteq [0,1]} \abs{\int_{R\x S \x T}
    W(x,y,z)
    \,dxdydz}.
\end{equation}
For each hypergraph $H$, one can easily extend the weak regularity lemma,
Lemma~\ref{lem:2-weak-reg}, to obtain a partition $\cQ$ of the vertex set of
$H$ into at most $2^{3/\e^2}$ parts so that $\norm{W^H - W^H_\cQ}_\square \leq
\e$ (regard $W^H$ as a 3-variable function for now, and $W^H_\cQ$ is derived
from $W$ by averaging over each cells induced by $\cQ$).
Theorem~\ref{thm:LS-analytic} also extends with virtually no change in the
proof. That is, allowing permutations of vertices, some subsequence of $H_n$
converges with respect to the vertex-cut norm \eqref{eq:3-vertex-cut} to a
3-variable symmetric function $\wt W \colon [0,1]^3 \to [0,1]$.

Unfortunately, the vertex-cut norm \eqref{eq:3-vertex-cut} is not strong enough
to guarantee a counting lemma. We want to say that if $H_1$ and $H_2$ are close with
respect to some cut norm, then $t(F, H_1)$ and $t(F, H_2)$ are close. If we
carry through the proof of Lemma~\ref{lem:2-count}, we find that $|t(F,H_1) -
t(F,H_2)| \leq e(F)\snorm{W^{H_1} - W^{H_2}}_\square$ holds when $F$ is a
\emph{linear hypergraph}, i.e., where every two edges of $F$ intersect in at
most one vertex. However, when $F$ is not linear, say $F = K_4^{(3)}$, then
this claim is completely false, as $t(F,H_1)$ and $t(F,H_2)$ can be separated
even when $\snorm{W^{H_1} - W^{H_2}}_\square$ is small. A counterexample for
$3$-uniform hypergraphs can be built by taking triangles of the random graph
$\GG(n,p)$, and then
keeping each triangle as a $3$-uniform edge with some probability
$q$. With
parameters $(p,q) = (1/2,1)$
and $(1,1/8)$, we obtain $3$-uniform hypergraphs that are close with respect to
the vertex-cut norm, and yet they have very different $K_4^{(3)}$ densities.

Now let us scrap the vertex-cut norm \eqref{eq:3-vertex-cut}. The proof of
the
counting lemma, Lemma~\ref{lem:2-count}, extends with respect to the
following
modified cut norm (again we use a 3-variable $W$ for now):
\begin{equation} \label{eq:2-cut}
  \text{(better cut norm)}\quad
  \norm{W}_{\square^2} = \sup_{\substack{u,v,w \colon [0,1]^2 \to [0,1] \\
      \text{symmetric}}} \abs{\int_{[0,1]^3} W(x,y,z)u(x,y)v(x,z)w(y,z)
    \,dxdydz}.
\end{equation}
For this cut norm, the counting lemma $|t(F,H_1) - t(F,H_2)| \leq
e(F)\snorm{W^{H_1} - W^{H_2}}_{\square^2}$ holds. However, like trying to
fit a
large rug in a small room, we quickly run into another issue: this norm is
too
strong and we do not have the compactness result corresponding to
Theorem~\ref{thm:LS-analytic}. Indeed, taking the sequence $H_n$ of
triangles of
$\GG(n,1/2)$ from \S\ref{sec:many-coord}, the two hypergraphs $H_n$ and $H_m$
are
typically not close with respect to $\norm{\cdot}_{\square^2}$, although
they
are close in homomorphism densities.

Even though we do not have compactness with respect to
$\norm{\cdot}_{\square^2}$,
we can still hope for a slightly weaker topology that gives
convergence of homomorphism densities. We can extend the weak regularity
lemma,
Lemma~\ref{lem:2-weak-reg}, to $\norm{\cdot}_{\square^2}$, where now
instead of
partitioning the vertex set $V = V(H)$, we partition the edges of the
underlying
complete graph $K_V = \binom{V}{2}$, i.e., the collection of unordered
pairs of
$V$. So now $\cQ$ is a partition $K_V = G_1 \cup \cdots \cup G_m$ of the
edges
of $K_V$ into $m$ graphs. The partition $\cQ$ of $K_V$ induces a partition
$\cQ^*$ on triples of vertices:
\[
(x,y,z) \sim_{\cQ^*} (x',y',z') \Leftrightarrow (x,y)\sim_\cQ (x',y'),\
(x,z)
\sim_\cQ (x',z'), \text{ and } (y,z) \sim_\cQ (y',z').
\]
Being somewhat sloppy with notation for the time being, we can form
$W^H_\cQ$
by averaging $W^H$ inside each cell of $\cQ^*$. Then the weak regularity
lemma
guarantees us a partition $\cQ$ of $K_V$ into at most $2^{3/\e^2}$ parts
so that
$\snorm{W^H - W^H_\cQ}_{\square^2} \leq \e$, and $|t(F,W^H) -
t(F,W^H_\cQ)| \leq
e(F)\e$ by the counting lemma.

For each hypergraph in the sequence $H_1, H_2, \dots$, apply the weak
regularity
lemma (for a uniform $\e$) to obtain a partition $\cQ_n$ of the complete
graph
on $V(H_n)$ into $m$ graphs: $K_{V(H_n)} = G_{n,1} \cup \cdots \cup
G_{n,m}$,
where $m$ depends on $\e$ but not on $n$.

By applying Theorem~\ref{thm:LS-analytic} on the graph sequence
$(G_{n,1})_{n\geq 1}$,
we can find a graphon $\wt Y_1 \colon [0,1]^2 \to [0,1]$ so that $G_{n,1}$
converges to $\wt Y_1$ as $n \to \infty$ along some subsequence . By further
restricting to
subsequences, we can find a $\wt Y_j$ for each $1 \leq j \leq m$ so that
$G_{n,j}$ converges to $\wt Y_j$ as $n \to \infty$ along a subsequence.

For each $n$, $\{G_{n,1},\dots, G_{n,m}\}$ is a partition of $K_{V(H_n)}$,
so
the same holds for the resulting limit\footnote{Provided that the limits
  of the
  various graph sequences are taken in a compatible way. This is a
  source of
  technical/notational annoyance later on, and it is the reason for
  introducing
  branching partitions in \S\ref{sec:branch}.}, in the sense that $\wt
Y_1 +
\cdots + \wt Y_m = 1$ almost everywhere as functions $[0,1]^2 \to [0,1]$.
Next we build a partition $\wt Q$ of the cube $[0,1]^3 = [0,1]^{r[2]}$
(coordinates indexed by $x_1, x_2, x_{12}$) by stacking together subsets
whose
heights are given by $\wt Y_j$. More precisely, $\wt Q = \{\wt Q_1, \dots,
\wt
Q_m\}$ where
\[
\wt Q_j = \{(x_1,x_2,x_{12}) \in [0,1]^3 : (\wt Y_1 + \cdots + \wt
Y_{j-1})(x_1,x_2) \leq x_{12} < (\wt Y_1 + \cdots + \wt Y_{j})(x_1,x_2)\}.
\]
This is the first place where the ``extra'' coordinates such as $x_{12}$
arise
even though we started with hypergraphs not requiring these extra
coordinates.
They arise because the limit graphon $\wt Y_1$ of a sequence of graphs
$G_{n,1}$
is not always a $\{0,1\}$-valued function.

The partition $\wt \cQ$ of $[0,1]^{r[2]}$ induces a partition  $\wt \cQ^*$
of
$[0,1]^6 = [0,1]^{r_<[2]}$:
\[
(x_1,x_2,x_3,x_{12},x_{13},x_{23}) \sim_{\wt \cQ^*}
(x'_1,x'_2,x'_3,x'_{12},x'_{13},x'_{23})
\Leftrightarrow (x_i,x_j,x_{ij}) \sim_{\wt \cQ} (x'_i,x'_j,x'_{ij}) \
\forall 1
\leq i < j \leq 3.
\]
The partition $\wt \cQ^*$ should not be viewed as a regularization
partition for
any $H_n$ (indeed, the extra coordinates do not even appear in $H_n$).
Instead,
the partitions $\cQ_n$ themselves become increasing close to $\wt \cQ$.
There is
a correspondence of cells of $\cQ_n$ with those of $\wt \cQ$, and this
induces a
correspondence between cells of $\cQ^*_n$ with those of $\wt \cQ^*$.

Now we construct the first limiting hypergraphon $\wt U_1$ as a step
function
$[0,1]^6 \to [0,1]$ that is constant on each part of $\wt \cQ^*$. On each
part
of $\wt \cQ^*$, we assign to $\wt U_1$ the limiting value of the average of
$W_n$ on the corresponding cell of $\cQ^*_n$, limit taken as $n\to\infty$
along
a further restricted subsequence. We have constructed $\wt U_1$, which
plays a
similar role as $\wt U_1$ near the end of \S\ref{sec:graphs}.

However, unlike \S\ref{sec:graphs}, $\wt U_1$ is not close in
$\norm{\cdot}_{\square^2}$ to $H_n$ for large $n$. It is a limit in the
following sense: we first $\e$-regularized $H_n$, and then took the graph
limit
of the partitions, created a new partition of $[0,1]^6$ using these lower
order
limits, and then constructed a step-function $U_1$ using this limiting
partition
and the limiting values on the steps. We knew from the earlier counting
lemma
(referred to later on as \emph{Counting Lemma I}) that
\begin{equation} \label{eq:3-count1}
  \abs{t(F,H) - t(F,W^H_{\cQ_n})} \leq e(F)\e.
\end{equation}
By what we will call \emph{Counting Lemma II}, we have (here $n \to \infty$
along
a subsequence)
\begin{equation} \label{eq:3-count2}
  \lim_{n\to\infty} t(F,W^H_{\cQ_n}) = t(F,\wt U_1).
\end{equation}
Here is some intuition why \eqref{eq:3-count2} holds. Both $W^H_{\cQ_n}$
and
$\wt U_1$ are step functions. We can split them up into weighted sums of
indicator functions, on which the claim reduces to checking homomorphism
densities for the graphons corresponding to parts of the partitions
$\cQ_n$ and
$\wt \cQ$. We know that the graphs which are the parts of $\cQ_n$ converge
to
the graphons from which $\wt \cQ$ is built. So the graph homomorphism
densities
converge.

This shows that $\wt U_1$ is a $O(e(F)\e)$-approximation to a
subsequence of $H_n$ in terms of $F$-densities. Now, take a smaller
$\e' < \e$, and build another $\wt U_2$, where the new partitions
$\cQ_n$ are refinements of the previous ones. Continuing this process,
we obtain a sequence $\wt U_1, \wt U_2, \dots$ which is a martingale
as before. The Martingale Convergence Theorem gives a pointwise almost
everywhere limit $\wt W$ of $\wt U_s$, $s \to \infty$, and $\wt W$ is
the desired limit.

\medskip

In proving 3-uniform hypergraph limits, we used the existence of graph
limits.  In general, we prove the existence of $k$-uniform hypergraph
limits by induction on $k$. There are a few further technical
difficulties. For example, we need to make sure that the limit of a
sequence of partitions remains a partition, so the limit needs to be
taken in a compatible way. Since we are working with multiple
partitions, we will need to deal with homomorphisms from $F$ to a vector
of hypergraphons, where the edges of $F$ individually land in
different hypergraphons. The details are addressed in the rest of this
paper.

\section{Notation} \label{sec:notation}

One (not so trivial) source of difficulty in working with hypergraphs is
the
complexity of notation. This section collects some of the notation and
conventions used in the rest of this paper. Some notations were already
introduced in \S\ref{sec:intro}.

\medskip

We shall omit the word ``measurable'' as everything we consider is assumed
to be
measurable.

\subsection{Hypergraphs} A $k$-uniform hypergraph $F$ is some finite
collection
of $k$-element subsets of some ground set, which we denote by $V(F)$. So
when we
talk about an element of $F$, we mean an edge of $F$, and $\abs{F}$ means
the
number of edges of $F$.

\subsection{Subsets, partitions, and hypergraphons}

\begin{definition}[Symmetric sets and partitions]
  A \emph{symmetric (measurable) subset} of $[0,1]^{r[k]}$ is one which
  is
  invariant under the action of all permutations of $[k]$. A
  \emph{symmetric
    (measurable) partition} of $[0,1]^{r[k]}$ is a partition of
  $[0,1]^{r[k]}$ into
  a finite collection of symmetric subsets.
\end{definition}

A symmetric subset $P \subseteq
[0,1]^{r[k]}$ is associated to a $k$-hypergraphon $W^P \colon
[0,1]^{r_<[k]}\to[0,1]$ by
integrating out the top coordinate:
\begin{equation}
  \label{eq:W^P}
  W^P(\bx_{r_<[k]}) := \int_0^1 1_P(\bx_{r[k]}) \,dx_{[k]}.
\end{equation}
For example, for $k = 3$, we have $P \subseteq [0,1]^3$, with coordinates
indexed by $r[2] = \{1,2,12\}$, and
\[
W^P(x_1, x_2) = \int_0^1 1_P(x_1,x_2,x_{12}) \,dx_{12}.
\]
This operation collapses the final coordinate in $P$. It will be helpful to
think of $P$ and $W^P$ as representing the same object. For example, when
$k=2$
this means we do not care how $P$ is placed along the $x_{12}$ coordinate,
as we
only care about how much $P$ intersects line segments of the form $\{x_1\}
\x
\{x_2\} \x [0,1]$. And conversely, for given a $W \colon [0,1]^2 \to
[0,1]$,
there are many $P \subseteq [0,1]^2$ satisfying $W^P = W$, e.g., any set
of the
form $P = \{(x,y,z) : a(x,y) \leq z \leq b(x,y)\}$ where $b(x,y) - a(x,y) =
W(x,y)$.

\subsection{Homomorphism densities} \label{sec:not-hom}
For any tuple of $k$-uniform hypergraphons $\bW = (W_1, \dots, W_m)$,
any $k$-uniform hypergraph $F$, and any map $\alpha \colon F \to [m]$,
define
the homomorphism density
\[
t_\alpha(F, \bW) := \int_{[0,1]^{r(V(F),k-1)}} \prod_{e \in F}
W_{\alpha(e)}(\bx_{r_<(e)}) \,d\bx.
\]

\begin{example} If $k = 2$, $F = K_3 = \{12,13,23\}$, $\alpha = (12
  \mapsto 1,
  13 \mapsto 2, 23 \mapsto 3)$, then
  \[
  t_\alpha(F, \bW) = \int_{[0,1]^3} W_1(x_1,x_2)W_2(x_1,x_3)W_3(x_2,x_3)
  \,dx_1dx_2dx_3
  \]
\end{example}

For any symmetric partition $\cP = (P_1, \dots, P_m)$ of
$[0,1]^{r[k]}$, define
\begin{equation} \label{eq:W^part}
  \bW^\cP := (W^{P_1}, \dots, W^{P_m})
  \qquad\text{and}\qquad
  t_\alpha(F, \cP) := t_\alpha(F, \bW^\cP).
\end{equation}

\subsection{Quotient and stepping operators} \label{sec:quotient}
Let $W \colon [0,1]^{r_<[k]} \to [0,1]$ be a $k$-uniform hypergraphon and
$\cQ$
a symmetric partition of $[0,1]^{r[k-1]}$ into $q$ parts $Q_1, Q_2, \dots,
Q_q
\subseteq [0,1]^{r[k-1]}$. The \emph{quotient} $W / \cQ$ is a $2q^k$-tuple
of
numbers in $[0,1]$ defined by assigning to each $k$-tuple $f = (f_1,
\dots, f_k)
\in [q]^k$ a pair $(v_f, w_f)$, referred to as (volume, average), as
follows:
\begin{itemize}
\item Volume: $v_f$ equals the integral
  \begin{equation} \label{eq:vol}
    v_f := \int_{\bx \in [0,1]^{r_<[k]}}
    1_{Q_{f_1}}(\bx_{r([k]\setminus\{1\})})
    1_{Q_{f_2}}(\bx_{r([k]\setminus\{2\})}) \cdots
    1_{Q_{f_k}}(\bx_{r([k]\setminus\{k\})}) \,d\bx.
  \end{equation}
\item Average: If $v_f = 0$, then we set $w_f = 0$. Otherwise, $w_f$
  is defined
  to be
  \begin{equation} \label{eq:avg}
    w_f := \frac{1}{v_f} \int_{\bx \in [0,1]^{r_<[k]}} W(\bx_{r_<[k]})
    1_{Q_{f_1}}(\bx_{r([k]\setminus\{1\})})
    1_{Q_{f_2}}(\bx_{r([k]\setminus\{2\})})
    \cdots 1_{Q_{f_k}}(\bx_{r([k]\setminus\{k\})}) \,d\bx.
  \end{equation}
\end{itemize}
Intuitively, the partition $\cQ$ induces a partition $\cQ^*$ of
$[0,1]^{r[k]}$
into parts enumerated by $f \in [q]^k$. Each cell of $\cQ^*$ has a volume
$v_f$
and an average value $w_f$ of $W$ on the cell.

If we have another $k$-uniform hypergraphon $W'$, and a symmetric partition
$\cQ'$ of $[0,1]^{r[k-1]}$ into $q$ parts ($\cQ$ and $\cQ'$ have the same
number
of parts) with volumes and weights $(v'_f, w'_f)$, we define
\begin{equation} \label{eq:quotient-d1}
  d_1(W/\cQ, W'/\cQ') := \sum_{f \in [q]^k} (|v_f - v'_f| + |v_f w_f - v'_f
  w'_f|).
\end{equation}
For any symmetric subset $P \subseteq [0,1]^{r[k]}$, we write
\[
P / \cQ := W^P / \cQ.
\]

A \emph{$\cQ$-step function} $U \colon [0,1]^{r_<[k]} \to \RR$ is a
function of
the form
\begin{equation}\label{eq:Q-step}
  U(\bx) = \sum_{f = (f_1, \dots, f_k) \in [q]^k} u_f
  1_{Q_{f_1}}(\bx_{r([k]\setminus\{1\})})
  1_{Q_{f_2}}(\bx_{r([k]\setminus\{2\})})
  \cdots 1_{Q_{f_k}}(\bx_{r([k]\setminus\{k\})})
\end{equation}
for some real values $u_f$. Since $\cQ$ is a partition, the indicator
functions in \eqref{eq:Q-step} all have disjoint support, which
together partition the domain $[0,1]^{r_<[k]}$. Usually $U$ is a
symmetric function, which is equivalent to having an additional
symmetry constraint on $u_f$, namely that $u_f = u_{f'}$ whenever $f'$
is obtained from $f'$ by a permutation of the coordinates.

The \emph{$\cQ$-stepping operator}, denoted by a subscript $\cQ$,
turns a $k$-uniform hypergraphon $W$ into a symmetric $\cQ$-step
function $W_\cQ$ by averaging over each induced cell of $\cQ^*$. More
precisely, we define $W_\cQ \colon [0,1]^{r_<[k]} \to [0,1]$ to be
(using $v_f$ and $w_f$ from $W/\cQ$ defined earlier)
\[
W_\cQ(\bx) := \sum_{f = (f_1, \dots, f_k) \in [q]^k} w_f
1_{Q_{f_1}}(\bx_{r([k]\setminus\{1\})})
1_{Q_{f_2}}(\bx_{r([k]\setminus\{2\})})
\cdots 1_{Q_{f_k}}(\bx_{r([k]\setminus\{k\})})
\]

We can also apply the stepping operator to a tuple of hypergraphons. If
$\bW =
(W_1, \dots, W_m)$, then
\[
\bW_\cQ := ((W_1)_\cQ, \dots, (W_m)_\cQ).
\]
In particular, if $\cP = \{P_1, \dots, P_m\}$ is a partition of
$[0,1]^{r[k]}$,
then we write
\[
\bW^\cP_\cQ := ((W^{P_1})_\cQ, \dots, (W^{P_m})_\cQ) = (W^{P_1}_\cQ, \dots,
W^{P_m}_\cQ)
\]

\subsection{Cut norm}

\begin{definition} \label{def:k-cut}
  For any symmetric function $W \colon [0,1]^{r_<[k]} \to
  \RR$, define
  \begin{equation} \label{eq:k-cut}
    \norm{W}_{\square^{k-1}}
    :=  \sup_{\substack{u_1, \dots, u_k \colon [0,1]^{r[k-1]} \to
        [0,1] \\ \text{symmetric}}}
    \abs{\int_{[0,1]^{r_<[k]}} W(\bx_{r_<[k]}) \prod_{i=1}^k
      u_i(\bx_{r([k]\setminus\set{i})}) \,d\bx}.
  \end{equation}
\end{definition}
Note that by linearity of the expression inside the absolute value in
\eqref{eq:k-cut}, it suffices to consider functions $u_i$'s which are
indicator
functions $1_{B_i}$ of symmetric subsets $B_i \subseteq [0,1]^{r[k-1]}$.
The
usual cut norm corresponds to the case $k=2$. The following example shows
$k=3$.

\begin{example}
  \label{ex:3-cut}
  For any symmetric function $W \colon [0,1]^{r_<[3]} \to \RR$,
  $\norm{W}_{\square^2}$ equals to
  % \begin{multline*}
  %   \norm{W}_{\square^2}
  %   = \sup_{u_1,u_2,u_3} \Bigg\vert \int_{[0,1]^6}
  % W(x_1,x_2,x_3,x_{12},x_{13},x_{23}) u_1(x_2,x_3,x_{23})
  % u_2(x_1,x_3,x_{13}) u_3(x_1,x_2,x_{12}) \\
  % dx_1dx_2dx_3dx_{12}dx_{13}dx_{23} \Bigg\vert
  % \end{multline*}
  \[
  \sup_{u_1,u_2,u_3} \abs{\int_{[0,1]^6}
    W(x_1,x_2,x_3,x_{12},x_{13},x_{23}) u_1(x_2,x_3,x_{23})
    u_2(x_1,x_3,x_{13}) u_3(x_1,x_2,x_{12}) \,dx_1dx_2dx_3dx_{12}dx_{13}dx_{23}}
  \]
  where $u_1,u_2,u_3$ vary over all symmetric functions
  $[0,1]^{r[2]} \to [0,1]$.
\end{example}

\section{Regularity and counting lemmas} \label{sec:reg-count}

\begin{definition} \label{def:W-reg}
  Let $W$ be a $k$-uniform hypergraphon and $\cQ$ a
  symmetric partition of $[0,1]^{r[k-1]}$. We say that $(W, \cQ)$ is
  \emph{weakly $\e$-regular} if $\norm{W - W_\cQ}_{\square^{k-1}} \leq
  \e$.

  For a symmetric subset $P \subseteq [0,1]^{r[k]}$, we say that $(P,
  \cQ)$ is
  weakly $\e$-regular if $(W^P, \cQ)$ is.
\end{definition}

\begin{lemma}[Weak regularity lemma] \label{lem:weak-reg}
  Let $k \geq 2$ and $\e > 0$. Let $\bW = (W_1, \dots, W_m)$ be a tuple
  of
  $k$-uniform
  hypergraphons.
  Let $\cQ$ be a
  symmetric partition of $[0,1]^{r[k-1]}$. Then there exists a
  partition $\cQ'$ refining $\cQ$ so that every part of $\cQ$ is refined
  into exactly $\lceil 2^{km/\e^2} \rceil$ parts (allowing empty
  parts) so that $(W_i, \cQ')$ is weakly $\e$-regular for every $1 \leq
  i \leq m$.
\end{lemma}

\begin{proof}
  We build the partition incrementally, starting with $\cQ$. At
  a given stage, suppose the partition is $\cR$. If $(W_i, \cR)$ is
  weakly
  $\e$-regular for every $i$ then we stop. Otherwise
  there is some $i$ with $\norm{W_i - (W_i)_\cR}_{\square^{k-1}} > \e$, so there
  exists
  symmetric subsets $B_1, \dots,
  B_k \subseteq [0,1]^{r([k-1])}$ such that
  \begin{equation}\label{eq:cut-dev}
    \abs{\int_{[0,1]^{r_<[k]}} (W_i - (W_i)_\cR)(\bx_{r_<([k])})
      \prod_{i=1}^k
      1_{B_i}(\bx_{r([k]\setminus\set{i})}) \,d\bx} > \e.
  \end{equation}
  Let $B \colon [0,1]^{r_<[k]} \to [0,1]$ be the function (not
  necessarily
  symmetric)
  \[
  B(\bx) := \prod_{i=1}^k 1_{B_i}(\bx_{r([k]\setminus\set{i})}) \,d\bx.
  \]
  For two functions $U, U' \colon [0,1]^{r_<[k]} \to [0,1]$, define the
  inner
  product
  \[
  \ang{U,U'} = \int_{[0,1]^{r_<[k]}} U(\bx) U'(\bx) \,d\bx.
  \]

  We will use the following easy fact: if $U'$ is a $\cQ$-step function,
  then
  $\ang{U, U'} = \ang{U_\cQ, U'}$.

  Now let $\cR'$ be the
  the minimal partition refining $\cR$ and $B_1, \dots, B_k$. Since
  $((W_i)_{\cR'})_{\cR} = (W_i)_\cR$, applying the fact above, we obtain
  \begin{equation} \label{eq:W-R'-R}
    \ang{(W_i)_{\cR'}, (W_i)_{\cR}} = \ang{(W_i)_{\cR}, (W_i)_{\cR}}
  \end{equation}
  Since $B$ is an $\cR'$-step function, we have $\ang{(W_i)_{\cR'}, B} =
  \ang{W_i,B}$. So by \eqref{eq:cut-dev}
  \begin{equation} \label{eq:W-R'-R-B}
    \abs{\ang{(W_i)_{\cR'} - (W_i)_{\cR}, B}} = \abs{\ang{W_i -
        (W_i)_{\cR}, B}} >
    \e.
  \end{equation}
  Since $\norm{U}_2^2 = \ang{U,U}$ for any $U$, we obtain by
  \eqref{eq:W-R'-R},
  the Cauchy-Schwarz inequality, and \eqref{eq:W-R'-R-B}
  \begin{equation} \label{eq:L2-inc}
    \norm{(W_i)_{\cR'}}_2^2 - \norm{(W_i)_{\cR}}_2^2=  \norm{(W_i)_{\cR'} -
      (W_i)_{\cR}}_2^2
    \geq \abs{\ang{(W_i)_{\cR'} - (W_i)_{\cR}, B}}^2 > \e^2.
  \end{equation}
  Furthermore, for every $1 \leq j \leq m$, $\norm{(W_j)_{\cR'}}_2^2 \geq
  \norm{(W_j)_{\cR}}_2^2$ by convexity since $((W_j)_{\cR'})_{\cR} =
  (W_j)_\cR$.

  The quantity $\norm{(W_1)_\cR}_2^2 + \cdots + \norm{(W_m)_\cR}_2^2$ is
  at most
  $m$, and each iteration above increases the sum by at least $\e^2$. So
  there can
  be
  at most $m/\e^2$ iterations. At the end we obtain a partition $\cQ'$
  so that
  $(W_i, \cQ')$ is weakly $\e$-regular for every $1 \leq i \leq m$. Each
  time we
  introduced at most $k$ new sets to refine
  the partition, so $\cR'$ refines each part of $\cR$ into at most $2^k$
  subparts. After at most $m/\e^2$ iterations,
  each part of the original partition $\cQ$ is refined into at most
  $2^{km/\e^2}$ parts. We can throw in some empty parts so that each
  part of $\cQ$
  is refined into exactly $\lceil 2^{km/\e^2} \rceil$ parts.
\end{proof}

\begin{lemma}[Counting lemma I] \label{lem:count1}
  Let $\bU = (U_1, \dots, U_m)$ and $\bW = (W_1, \dots, W_m)$ be two
  $m$-tuple of
  $k$-uniform hypergraphons and $\cQ$ a
  symmetric partition of $[0,1]^{r([k-1])}$. Suppose that $\norm{W_i -
    U_i}_{\square^{k-1}} \leq \e$ for each $i$. Then for any
  $k$-uniform
  hypergraph $F$ and any
  map $\alpha \colon F \to [m]$, we have
  \[
  \abs{t_\alpha(F, \bU) - t_\alpha(F, \bW)} \leq \abs{F} \e.
  \]
\end{lemma}

\begin{proof}
  Let $V = V(F)$ and $F = \{e_1, \dots,
  e_{\abs{F}}\}$. Write as a telescoping sum
  \begin{align*}
    &t_\alpha(F,\bU) - t_\alpha(F,\bW)\\
    &=
    \int_{[0,1]^{r(V,k-1)}} \paren{\prod_{i=1}^{\abs{F}}
      U_{\alpha(e_i)}(\bx_{r_<(e_i)}) - \prod_{i=1}^{\abs{F}}
      W_{\alpha(e_i)}(\bx_{r_<(e_i)})} \,d\bx_{r(V,k-1)}
    \\
    &= \sum_{j=1}^{\abs{F}}  \int_{[0,1]^{r(V,k-1)}}
    \paren{\prod_{i=1}^{j-1} U_{\alpha(e_i)}(\bx_{r_<(e_i)})}
    (U_{\alpha(e_j)}- W_{\alpha(e_j)})(\bx_{r_<(e_j)})
    \paren{\prod_{i=j+1}^{\abs{F}} W_{\alpha(e_i)}(\bx_{r_<(e_i)})} \
    d\bx.
  \end{align*}
  The $j$-th term in the final sum is bounded by $\snorm{U_{\alpha(e_j)}
    -
    W_{\alpha(e_j)}}_{\square^{k-1}} \leq \e$. Indeed, if we fix all
  variables other than $\bx_{r_<(e_j)}$, then all the factors except for
  $(U_{\alpha(e_j)}- W_{\alpha(e_j)})(\bx_{r_<(e_j)}) $
  have the form $u(\bx_{r(f)})$ for some $f \subsetneq e_j$, where $f$
  is the intersection of $e_j$ with another edge $e_{j'}$. So the
  the integral can be bounded by the $(k-1)$-cut norm, as claimed.
\end{proof}

\begin{lemma}[Counting lemma II] \label{lem:count2}
  Let $\bU = (U_1, \dots, U_m)$ and $\bW = (W_1, \dots, W_m)$ be two
  $m$-tuples of
  $k$-uniform hypergraphons. Let
  $\cQ = \{Q_1, \dots, Q_q\}$ and $\cR = \{R_1, \dots, R_q\}$ be
  symmetric partitions of $[0,1]^{r[k-1]}$. Suppose that $d_1(U_i/\cQ,
  W_i/\cR)
  \leq \delta$ for each $i$. Then for any $k$-uniform hypergraph $F$ and
  any
  map $\alpha \colon F \to [m]$,
  \[
  \abs{t_\alpha(F, \bU_\cQ) - t_\alpha(F, \bW_{\cR})} \leq \abs{F}\delta
  +
  \sum_{\beta \colon \del F \to
    [q]} \abs{t_\beta(\del F, \cQ) - t_\beta(\del F, \cR)},
  \]
  where the sum is taken over all maps $\beta \colon \del F \to [q]$,
  and $\del F$
  is the
  $(k-1)$-uniform hypergraph on $V(F)$ consisting of $(k-1)$-element
  subsets of
  $V(F)$
  that are contained in some edge of $F$.
\end{lemma}

\begin{proof}
  We can replace each $U_i$ by $(U_i)_\cQ$ as this does not change
  $U_i/\cQ$ or
  $t_\alpha(F, \bU_\cQ)$. So we may assume that every $U_i$ is a
  symmetric
  $\cQ$-step function, i.e., $\bU_\cR = \bU$. Similarly, assume that every
  $W_i$
  is a symmetric $\cR$-step function.

  For each $f \in [q]^k$, let $(v_{i,f}, w_{i,f})$ denote the volume and
  average
  corresponding to $f$ in $U_i/\cQ$, and let $(v'_{i,f}, w'_{i,f})$
  denote the
  same for $W_i/\cR$.

  For each $1 \leq i \leq m$, construct a symmetric $\cQ$-step function
  $U'_i$
  from $U_i$ by changing its value on the step corresponding to $f$
  from
  $w_{i,f}$ to $w'_{i,f}$. So $U'_i/\cQ$ has $(v_{i,f}, w'_{i,f})$ as
  its volumes
  and averages. In other words,
  \begin{align}
    U_i(\bx_{r_<[k]}) &= \sum_{f = (f_1, \dots, f_q) \in [q]^k} w_{i,f}
    1_{Q_{f_1}}(\bx_{r([k]\setminus\{1\})}) \cdots
    1_{Q_{f_k}}(\bx_{r([k]\setminus\{k\})}); \label{eq:count2-U} \\
    U'_i(\bx_{r_<[k]}) &= \sum_{f = (f_1, \dots, f_q)\in [q]^k}
    w'_{i,f}
    1_{Q_{f_1}}(\bx_{r([k]\setminus\{1\})}) \cdots
    1_{Q_{f_k}}(\bx_{r([k]\setminus\{k\})}); \label{eq:count2-U'} \\
    W_i(\bx_{r_<[k]}) &= \sum_{f = (f_1, \dots, f_q)\in [q]^k} w'_{i,f}
    1_{R_{f_1}}(\bx_{r([k]\setminus\{1\})}) \cdots
    1_{R_{f_k}}(\bx_{r([k]\setminus\{k\})}). \label{eq:count2-R}
  \end{align}
  Write $\bU' = (U'_1, \dots, U'_m)$. We have
  \begin{align*}
    \norm{U_i - U'_i}_1 &= \sum_{f \in [q]^k} v_{i,f}|w_{i,f} -
    w'_{i,f}|
    \leq \sum_{f \in [q]^k} (|v_{i,f}w_{i,f} -
    v'_{i,f}w'_{i,f}| + w'_{i,f}|v'_{i,f} - v_{i,f}|) \\
    &\leq \sum_{f \in [q]^k} (|v_{i,f}w_{i,f} -
    v'_{i,f}w'_{i,f}| + |v'_{i,f} - v_{i,f}|)
    = d_1(U_i/\cQ, W_i/\cR) \leq \delta.
  \end{align*}
  So $\norm{U_i - U'_i}_1 \leq \delta$ for each $i$. It follows that
  \begin{equation} \label{eq:count2-U-U'}
    \abs{t_\alpha(F, \bU) - t_\alpha(F, \bU')} \leq \abs{F}\delta.
  \end{equation}
  (This follows from Counting Lemma I, but it's in fact even easier.)
  From \eqref{eq:count2-U'} we have
  \begin{align}
    t_\alpha(F, \bU') &= \int_{[0,1]^{r(V(F),k-1)}} \prod_{e = \{j_1,
      \dots, j_k\}
      \in F} \paren{\sum_{f = (f_1, \dots, f_q) \in [q]^k}
      w'_{\alpha(e),f}
      1_{Q_{f_1}}(\bx_{r(e\setminus\{j_1\})}) \cdots
      1_{Q_{f_k}}(\bx_{r(e\setminus\{j_k\})}) } \,d\bx \nonumber
    \\
    &= \sum_{\beta \colon \del F \to [q]} \paren{\prod_{e \in F}
      w'_{\alpha(e),\beta(\del e)}} t(\del F, \cQ).
    \label{eq:count2-U'-expand}
  \end{align}
  Here $\beta(\del e) = (\beta(e \setminus \{j_1\}), \dots,
  \beta(e\setminus\{j_k\})) \in [q]^k$ when $e = \{j_1, \dots,
  j_k\}$. The last equality above needs some pondering. Essentially we
  expand the product of sums in the previous line and note that since
  $\cQ$ is a partition, the nonzero terms in the expansion correspond
  to assigning an $f$ to every $e$ in a compatible way: if two edges
  $e = \{j_1,\dots, j_k\}$ and $e' = e\cup\{j'_l\} \setminus\{j_l\}$
  intersect in exactly $k-1$ vertices, and $f$ is assigned to $e$, and
  $f'$ is assigned to $e'$, then $f_l = f'_l$. These assignments are
  in bijection with $\beta \colon \del F \to [q]$, where $\beta$
  corresponds to the assignment assigning $e$ to $\beta(\del e)$.

  Similar to \eqref{eq:count2-U'-expand} we have
  \begin{equation} \label{eq:count2-W-expand}
    t_\alpha(F, \bW) = \sum_{\beta \colon \del F \to [q]} \paren{\prod_{e
        \in F}
      w'_{\alpha(e),\beta(\del e)}} t(\del F, \cR).
  \end{equation}
  Combing \eqref{eq:count2-U'-expand} and \eqref{eq:count2-W-expand}
  using the
  triangle inequality and noting that $0 \leq w'_{i,f} \leq 1$, we have
  \begin{equation} \label{eq:count2-U'-W}
    \abs{t_\alpha(F, \bU') - t_\alpha(F, \bW)} \leq
    \sum_{\beta \colon \del F \to
      [q]} \abs{t_\beta(\del F, \cQ) - t_\beta(\del F, \cR)}.
  \end{equation}
  The lemma follows from combining \eqref{eq:count2-U-U'} and
  \eqref{eq:count2-U'-W} using the triangle inequality.
\end{proof}

\section{Branching partitions} \label{sec:branch}

Now we are almost ready to build the limiting object. We will proceed
by induction on $k$ (for $k$-uniform hypergraphons). The situation is
very simple when $k = 1$, since in this case a hypergraphon is simply
a number between 0 and 1. To build the limiting hypergraphon in
general, we will need to repeatedly apply the weak regularity lemma to
obtain a refining chain of partitions.  Since we need to apply
induction on $k$, we need to have a stronger induction hypothesis that
involves a sequence of not just single hypergraphons, but refining
chains of partitions. This motivates the following definition of a
branching partition, which is a special case a \emph{filtration}, in
the language of probability. See Figure~\ref{fig:branch}.

\begin{definition}

  A \emph{degree} $p = (p_1, p_2, \dots) \in \NN^\NN$ \emph{(symmetric) branching
    partition} $\sP$ of $[0,1]^{r[k]}$ is a collection of symmetric
  subsets $P_i$ of $[0,1]^{r[k]}$, collected into \emph{levels}, where
  each level $\cP_l$ is a symmetric partition of $[0,1]^{r[k]}$:
  \begin{itemize}
  \item Level 0: $\cP_0 = \{ [0,1]^{r[k]} \}$
  \item Level 1: $\cP_1 = \{P_1, P_2, \dots,
    P_{p_1}\}$ is a symmetric partition of $[0,1]^{r[k]}$.
  \item Level $l$ ($l \geq 2$): $\cP_l$ is a refinement of
    $\cP_{l-1}$, where
    each part of $\cP_{l-1}$ gets further refined into exactly $p_l$
    parts.
  \end{itemize}
  An \emph{index} at level $l$ is a tuple $i = (i_1, i_2, \dots, i_l) \in
  [p_1] \x
  [p_2] \x \cdots \x [p_l]$, which points to the symmetric subset $P_i =
  P_{i_1,
    \dots, i_l} \in \cP_l$ at level $l$, where $P_i$ is the $i_l$-th
  part in the
  refinement of the part $P_{i_1, \dots,i_{l-1}}$ at level $l-1$,
  whenever $l
  \geq 2$
  (all partitions are ordered).
\end{definition}

\begin{figure}
  \begin{tikzpicture}[yscale=1]
    \draw [decorate,decoration={brace,amplitude=.2cm}] (-7.3,1)  --
    (-7.3,4.5)
    node [midway,xshift=-.6cm] {$\sP$};
    \node at (-6,4.2) {Level 0: $\cP_0$};
    \node at (-6,3) {Level 1: $\cP_1$};
    \node at (-6,2) {Level 2: $\cP_2$};
    \node at (-6,1.5) {\vdots};
    \node at (0,4.2) {$[0,1]^{r[k]}$};
    \node at (.2,3.5) {\tiny $p_1$};
    \draw (.3,3.9) arc (0:-180:.3);
    \draw (0,3.9) -- +(-3,-.6);
    \draw (0,3.9) -- +(-1,-.6);
    \draw (0,3.9) -- +(1,-.6);
    \draw (0,3.9) -- +(3,-.6);
    \node at (-3,3) {$P_1$};
    \node at (-2.5,2.7) {\tiny $p_2$};
    \draw (-2.7,2.8) arc (0:-180:.3);
    \draw (-3,2.8) -- +(-.8,-.5);
    \draw (-3,2.8) -- +(0,-.5);
    \draw (-3,2.8) -- +(.8,-.5);
    \node at (-1,3) {$P_2$};
    \node at (-0.5,2.7) {\tiny $p_2$};
    \draw (-0.7,2.8) arc (0:-180:.3);
    \draw (-1,2.8) -- +(-.5,-.5);
    \draw (-1,2.8) -- +(0,-.5);
    \draw (-1,2.8) -- +(.5,-.5);
    \node at (1,3) {$\cdots$};
    \node at (3,3) {$P_{p_1}$};
    \draw (3.3,2.8) arc (0:-180:.3);
    \draw (3,2.8) -- +(-.5,-.5);
    \draw (3,2.8) -- +(0,-.5);
    \draw (3,2.8) -- +(.5,-.5);
    \node at (3.5,2.7) {\tiny $p_2$};
    \node at (-3.8,2) {$P_{1,1}$};
    \draw (-3.5,1.8) arc (0:-180:.3); \node at (-3.3,1.6) {\tiny $p_3$};
    \draw (-3.8,1.8) -- +(-.3,-.5);
    \draw (-3.8,1.8) -- +(0,-.5);
    \draw (-3.8,1.8) -- +(.3,-.5);
    \node at (-3,2) {$\cdots$};
    \node at (-2.2,2) {$P_{1,p_2}$};
  \end{tikzpicture}
  \caption{A branching partition} \label{fig:branch}
\end{figure}
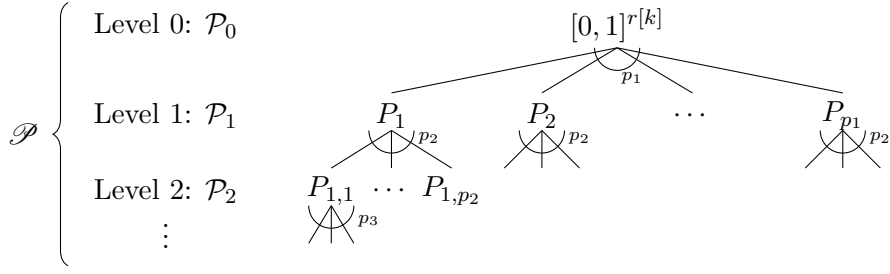

\noindent\emph{Font convention.} $\sP$ is a branching partition, $\cP$ is a
partition, and $P$ is a subset of $[0,1]^{r[k]}$.

\begin{example} \label{ex:W-branch}
  A symmetric subset $P \subseteq [0,1]^{r[k]}$ or a $k$-uniform
  hypergraphon $W$
  (related by \eqref{eq:W^P}) can be thought of as a degree
  $(2,1,1,1,\dots)$
  branching
  partition: level 1 is $P$ and $P^c$ (the complement of $P$ in
  $[0,1]^{r[k]}$) and all subsequent levels are trivial refinements.
\end{example}

We can generalize the notion of regularity from Definition~\ref{def:W-reg}
to
branching partitions as follows.

\begin{definition} \label{def:part-reg}
  Let $\sP$ be a branching partition of $[0,1]^{r[k]}$ and $\sQ$ a
  branching partition of $[0,1]^{r[k-1]}$. We say that $(\sP,\sQ)$ is
  \emph{weakly $(\e_1, \e_2, \dots)$-regular} if for every $s \geq 1$, whenever
  $P
  \subseteq [0,1]^{r[k]}$ is a member of $\sP$ of level at most $s$, and
  $\cQ_s$ is the level $s$ partition of $[0,1]^{r[k-1]}$ in $\sQ$,
  the pair $(P, \cQ_s)$ is weakly $\e_s$-regular.
\end{definition}

\begin{lemma}[Weak regularity lemma for branching partitions]
  \label{lem:branch-reg}
  For every $k \geq 2$, $p = (p_1, p_2, \dots) \in \NN^\NN$ and $\e =
  (\e_1, \e_2, \dots) \in \RR_{> 0}^{\NN}$, we can find a
  $q = (q_1, q_2, \dots) \in \NN^\NN$ so that the following holds: for
  every
  degree $p$
  branching partition $\sP$ of
  $[0,1]^{r[k]}$, there exists a degree $q$ branching partition $\sQ$
  of $[0,1]^{r[k-1]}$ so that $(\sP, \sQ)$ is weakly $\e$-regular.
\end{lemma}

\begin{proof}
  Take $q_s = \lceil 2^{k(p_1 + p_1p_2 + \cdots + p_1p_2\cdots p_{s})/\e_s^2} \rceil$. We build
  $\sQ$ successively by level. To obtain the level $s$ partition in
  $\sQ$,
  applying Lemma \ref{lem:weak-reg} with $\e = \e_s$, $\bW$ the
  collection of
  hypergraphons corresponding to all members of $\sP$ of level at most
  $s$, and
  $\cQ$ the level $s-1$ partition in $\sQ$.
\end{proof}

Now we introduce two notions of convergence for branching partitions. The first
notion, called \emph{left-convergence}, is based on convergence of homomorphism
densities. The second notion, called \emph{partitionable convergence}, is
based on convergence of regularity partitions. We will show, using our counting
lemmas, that partitionable convergence implies left-convergence.

\begin{notation}
  Given degree $p = (p_1, p_2, \dots)$ branching partitions
  $\sP_1, \sP_2, \dots$ and $\wt \sP$ of $[0,1]^{r[k]}$ and
  degree $q = (q_1, q_2, \dots)$ branching partitions
  $\sQ_1, \sQ_2, \dots$ and $\wt \sQ$ of $[0,1]^{r[k-1]}$, we use the
  following
  notation
  to refer to the partitions and parts in these branching partitions.
  \begin{itemize}
  \item For each $l \geq 1$, $\cP_{n,l}$ is  the level $l$ partition
    in $\sP_n$,
    and
    $\wt\cP_l$ is the level $l$ partition in $\wt\sP$.
  \item For each $s \geq 1$,
    $\cQ_{n,s}$ is the level $s$ partition in $\sQ_n$, and $\wt\cQ_s$
    is the level
    $s$ partition in $\wt\sQ$.
  \item For each index $i = (i_1, i_2, \dots, i_l) \in [p_1]
    \x \cdots \x [p_l]$, $P_{n,i}$ is the index $i$ element of $\sP_n$
    and $\wt
    P_i$
    is the index $i$ element of $\wt \sP$.
  \end{itemize}
\end{notation}

\begin{definition}[Left-convergence: $\sP_n \to \wt \sP$] We say that a
  sequence $\sP_1, \sP_2,\dots$ of degree $p$ branching
  partitions of $[0,1]^{r[k]}$ \emph{left-converges} to
  another degree $p$ branching partition $\wt \sP$ of $[0,1]^{r[k]}$,
  written
  $\sP_n \to \wt \sP$, if
  \begin{equation}\label{eq:left-conv-part}
    \lim_{n \to \infty} t_\alpha(F, \cP_{n,l}) = t_\alpha(F, \wt \cP_l)
    \qquad
    \text{for all } F, l, \alpha
  \end{equation}
  where $F$ ranges over all $k$-uniform hypergraphs, $l$ ranges over all
  positive
  integers, and $\alpha$ ranges over all maps $F \to [p_1 \cdots p_l]$. Recall
  from \eqref{eq:W^part} that $t_\alpha(F,\cP) := t_\alpha(F,\bW^{\cP})$ for a
  partition $\cP$.
\end{definition}

\begin{definition}[Partitionable convergence: $\sP_n \dashto
  \wt \sP$] \label{def:part-converge}
  We say that a
  sequence $\sP_1, \sP_2,\dots$ of degree $p=(p_1, p_2, \dots)$ branching
  partitions of $[0,1]^{r[k]}$ \emph{partitionably converges} to
  another degree $p$ branching partition $\wt \sP$ of $[0,1]^{r[k]}$,
  written $\sP_n \dashto \wt \sP$, if the following
  is satisfied (the definition is inductive on $k$).

  When $k = 1$, for every index $i = (i_1, \dots, i_l) \in [p_1] \x
  \cdots \x [p_l]$, we have $\lim_{n\to \infty} \lambda(P_{n,i}) =
  \lambda(\wt P_i)$, where $\lambda$ is the Lebesgue measure on
  $[0,1]$.

  When $k \geq 2$, there exists some $q \in \NN^\NN$ and degree $q$
  branching partitions $\sQ_1, \sQ_2, \dots$ and $\wt \sQ$ of
  $[0,1]^{r[k-1]}$ satisfying:
  \begin{itemize}
  \item[(a)] $(\sP_n, \sQ_n)$ is weakly $(1, 1/2, 1/3,
    \dots)$-regular for every
    $n$;
  \item[(b)] $\sQ_n \dashto \wt \sQ$ as $n \to \infty$ (defined inductively);
  \item[(c)] For every $s \geq 1$ and every index $i \in [p_1] \x
    \cdots \x
    [p_l]$, one
    has $\lim_{n\to\infty} d_1(P_{n,i}/\cQ_{n,s}, \wt P_i/\wt \cQ_s) = 0$;
  \item[(d)] For every member $\wt P \subseteq [0,1]^{r[k]}$ of $\wt
    \sP$, one has
    $(W^{\wt P})_{\wt\cQ_s} \to W^{\wt P}$ pointwise almost everywhere
    as $s \to
    \infty$.
  \end{itemize}
\end{definition}

\begin{lemma}[Partitionable convergence implies left-convergence]
  \label{lem:part-left}
  If
  $\sP_n \dashto \wt \sP$ then $\sP_n \to \wt\sP$.
\end{lemma}

\begin{proof}
  We use induction on $k$. When $k = 1$, the claim is trivial. Now
  assume $k
  \geq 2$.

  We need to show that \eqref{eq:left-conv-part} holds. Fix $F, l,
  \alpha$. Let $m
  = p_1 \cdots p_s$. Let $\sQ_n$ and $\wt\sQ$ be as in
  Definition~\ref{def:part-converge}, and let  $q=(q_1, q_2, \dots)$ be
  the degree
  of $\wt\sQ$.

  Let $\e > 0$. By Definition~\ref{def:part-converge}(d),
  $\bW^{\wt\cP_l}_{\wt\cQ_s}$ converges pointwise almost everywhere in
  each
  coordinate to $\bW^{\wt\cP_l}$ as $s \to \infty$, so $\lim_{s \to
    \infty}
  t_\alpha(F, \bW^{\wt\cP_l}_{\wt\cQ_s}) = t_\alpha(F, \bW^{\wt\cP_l})$.
  We can
  find an $s \geq \max\{l,\abs{F}/\e\}$ so that $|t_\alpha(F,
  \bW^{\wt\cP_l}_{\wt\cQ_s}) - t_\alpha(F, \cP_l)| \leq \e$. Fix this
  value of
  $s$.

  By Definition~\ref{def:part-converge}(b) we have $\sQ_n \dashto
  \wt\sQ$, so
  $\sQ_n \to \sQ$ by the induction hypothesis. Thus
  \begin{equation} \label{eq:part-left-Q-count}
    \lim_{n \to \infty} t_\beta(\del F, \cQ_{n,s}) = t_\beta(\del F,
    \wt\cQ_{s})
  \end{equation}
  for all  $\beta \colon \del F \to [q_1q_2\cdots q_s]$. See
  Lemma~\ref{lem:count2} for the definition of $\del F$. We have
  \begin{multline} \label{eq:part-left-split}
    |t_\alpha(F, \cP_{n,l}) - t_\alpha(F, \wt \cP_l)| \\
    \leq
    |t_\alpha(F, \cP_{n,l}) - t_\alpha(F, \bW^{\cP_{n,l}}_{\cQ_{n,s}})|
    +
    |t_\alpha(F, \bW^{\cP_{n,l}}_{\cQ_{n,s}}) - t_\alpha(F,
    \bW^{\wt\cP_{l}}_{\wt\cQ_{s}})|
    +
    |t_\alpha(F, \bW^{\wt\cP_l}_{\wt\cQ_s}) - t_\alpha(F, \wt \cP_l)|
  \end{multline}
  As $n \to \infty$, the first term on the right hand side of
  \eqref{eq:part-left-split} has a limsup of at most $\abs{F}/s \leq
  \e$ by Counting Lemma I (Lemma~\ref{lem:count1}) since
  $(P,\cQ_{n,s})$ is $1/s$-regular for every $P \in \cP_{n,l}$ by
  Definition~\ref{def:part-converge}(a).  The second term on the RHS
  of \eqref{eq:part-left-split} goes to zero by Counting Lemma II
  (Lemma~\ref{lem:count2}), Definition~\ref{def:part-converge}(c), and
  \eqref{eq:part-left-Q-count}. The third term on the RHS of
  \eqref{eq:part-left-split} is at most $\e$ using our choice of
  $s$. It follows that $\limsup_{n\to\infty} |t_\alpha(F, \cP_{n,l}) -
  t_\alpha(F, \wt \cP_l)| \leq 2\e$. Since $\e$ can be made arbitarily
  small, we obtain $\lim_{n\to\infty} t_\alpha(F, \cP_{n,l})=
  t_\alpha(F, \wt \cP_l)$ as desired.
\end{proof}

\begin{proposition} \label{prop:part-converge}
  Let $p \in \NN^\NN$. Let $\sP_1, \sP_2 \cdots$ be a sequence of degree
  $p$
  branching partitions of $[0,1]^{r[k]}$. Then there exists another
  degree $p$
  branching partition
  $\wt\sP$ of $[0,1]^{r[k]}$ so that $\sP_n \dashto \wt\sP$ as $n \to
  \infty$
  along some infinite subsequence.
\end{proposition}

\begin{proof}
  We use induction on $k$. The claim is easy when $k = 1$, since we can
  pick a
  subsequence so that for each index $i$, the measure $\lambda(P_{n,i})$
  converges
  to some value $a_i$ as $n \to \infty$, and we can take the limit
  $\wt\sP$ to be
  a branching partition where $\wt P_{i}$ is an interval with length
  $a_i$.

  Now assume $k \geq 2$. By
  Lemma~\ref{lem:branch-reg}, there exists a $q \in \NN^\NN$ so that for
  every $n$
  we can find a degree $q$ branching
  partition $\sQ_n$ of $[0,1]^{r[k-1]}$ so that $(\sP_n, \sQ_n)$ is
  weakly $(1,1/2,1/3,\dots)$-regular, thereby satisfying (a) in
  Definition~\ref{def:part-converge}. Applying the induction hypothesis,
  we can
  restrict to a
  subsequence so that $\sQ_n \dashto \wt \sQ$ for some branching
  partition $\wt \sQ$ of $[0,1]^{r[k-1]}$ (here and onwards in this proof we
  abuse
  notation by only considering convergence as $n \to \infty$ along some
  subsequence. We will be repeatedly taking subsequences, and the
  conclusion will
  follow by a standard diagonalization argument). So (b) is satisfied.

  By further restricting to a subsequence, we may assume that for each
  $s \geq 1$ and each index $i$, the quotient $P_{n,i}/\cQ_{n,s}$
  converges coordinate-wise as $n \to \infty$.  Let $W_{n,i} :=
  W^{P_{n,i}}$ be the hypergraphon associated to $P_{n,i}$.  Let $\wt
  W_{i,s} \colon [0,1]^{r_<[k]} \to [0,1]$ be a symmetric $\wt \cQ_s$-step
  function, with values assigned so that
  $d_1(W_{n,i}/\cQ_{n,s},\wt W_{i,s}/\wt \cQ_s) \to 0$ as $n \to \infty$. This
  is possible since we previously assumed that $P_{n,i}/\cQ_{n,s}$
  converges coordinatewise as $n \to \infty$, so that are now simply
  putting in the limiting values of the ``average'' coordinates into a
  template for a symmetric $\wt \cQ_s$-step function in order to construct
  $\wt W_{i,s}$. To see that the ``volume'' coordinates~\eqref{eq:vol}
  of $\cQ_{n,s}$ converge to those of $\wt \cQ_s$, note that this
  amount to the claim that $\lim_{n\to\infty}
  t_\beta(K_k^{(k-1)},\cQ_{n,s}) = t_\beta(K_k^{(k-1)},\wt\cQ_s)$ for
  every $\beta \colon K_k^{(k-1)} \to [q]$, where $K_k^{(k-1)}$ is the
  $(k-1)$-uniform simplex, i.e., the collection of all $(k-1)$-element
  subsets of $[k]$. The convergence of these homomorphism densities
  follows from $\sQ_n \to \wt \sQ$ which in turn follows from $\sQ_n
  \dashto \wt \sQ$ and Lemma~\ref{lem:part-left}.

  \medskip\noindent
  \emph{Claim 1.} $(\wt W_{i,s+1})_{\wt \cQ_s} = \wt W_{i,s}$.
  \smallskip

  \noindent\emph{Proof of Claim 1.} We have
  \begin{equation} \label{eq:limit-claim1a}
    \lim_{n \to \infty} d_1(W_{n,i}/\cQ_{n,s}, \wt W_{i,s} / \wt\cQ_s) = 0
  \end{equation}
  and
  \begin{equation}\label{eq:limit-claim1b}
    \lim_{n \to \infty} d_1(W_{n,i}/\cQ_{n,s+1}, \wt W_{i,s+1} /
    \wt\cQ_{s+1}) = 0
  \end{equation}
  Since $\wt \cQ_{s+1}$ is a refinement of $\wt \cQ_s$, by merging together
  parts in $W_{n,i}/\cQ_{s+1}$ and $\wt W_{i,s+1} / \wt \cQ_{s+1}$, we
  deduce from
  \eqref{eq:limit-claim1b}
  \begin{equation}\label{eq:limit-claim1c}
    \lim_{n \to \infty} d_1(W_{n,i}/\cQ_{n,s}, \wt W_{i,s+1} / \wt\cQ_{s})
    = 0
  \end{equation}
  From \eqref{eq:limit-claim1a} and \eqref{eq:limit-claim1c} we obtain
  $\wt
  W_{i,s+1} / \wt \cQ_{s} = \wt W_{i,s} / \wt \cQ_s$, which implies $(\wt
  W_{i,s+1})_{\wt \cQ_s} = \wt W_{i,s}$  since both sides are $\wt \cQ_s$-step
  functions
  $\square$ \medskip

  It follows that $\wt W_{i,1}, \wt W_{i,2}, \wt W_{i,3}, \dots$ is a
  martingale
  with respect to the filtration\footnote{To be more precise, let
    $[0,1]^{r[k]}$
    be the probability space equipped with the uniform Lebesgue
    measure. For each $s
    \geq 1$ let $\cB_s$ be the minimal $\sigma$-algebra on
    $[0,1]^{r[k]}$ generated
    by functions of the form $1_Q(\bx_{r([k]\setminus\{j\})})$ ranged
    over $Q \in
    \cQ_s$ and $j \in [k]$. Then $\wt W_{i,s}$ is a $\cB_s$-measurable
    random
    variable, and Claim 1 implies that $\wt W_{i,1}, \wt W_{i,2},
    \cdots$ is a
    martingale adapted to the filtration $\cB_1 \subseteq \cB_2
    \subseteq \cdots$}
  induced by $\wt \cQ_1, \wt \cQ_2, \dots$. By the Martingale
  Convergence Theorem,
  there exists some $\wt W_i$, so that $\wt W_{i,s} \to \wt W_i$
  pointwise almost
  everywhere as $s \to \infty$. Furthermore
  $(\wt W_i)_{\wt \cQ_s} = \wt W_{i,s}$.

  \medskip \noindent \emph{Claim 2.} Let $l \geq 1$, and let $i = (i_1, \dots,
  i_{l-1})
  \in [p_1] \x \cdots \x [p_{l-1}]$ an index at level $l-1$, which points to a
  part
  in $\cP_1$ that splits into indices
  $\{j_1, \dots, j_{p_l}\} = i \x [p_l]$ at level $l$. Then
  \[
  \wt W_{j_1} + \cdots + \wt W_{j_{p_l}} = \wt W_i \quad \text{almost
    everywhere}.
  \]

  \noindent\emph{Proof of Claim 2.} Since $P_{n,j_1}, \dots,
  P_{n,j_{p_s}}$  is a
  partition of $P_{n,i}$, we have
  \[
  W_{n,j_1} + \cdots +  W_{n,j_{p_s}} =  W_{n,i}
  \]
  Taking the $\cQ_{n,s}$ quotient of both sides and then take the limit
  as $n \to
  \infty$, we find the following equality for these $\wt \cQ_s$-step
  functions.
  \[
  \wt W_{j_1,s} + \cdots + \wt W_{j_s,s} = \wt W_{i,s}
  \]
  Taking $s \to \infty$ and using the pointwise almost everywhere
  convergence of
  $\wt W_{j,s} \to \wt W_{j}$ as $s \to \infty$ for every index $j$, we
  obtain
  Claim 2. $\square$ \medskip

  Claim 2 tells us that we can find a branching partition $\wt \sP$ of
  $[0,1]^{r[k]}$ so that the part $\wt P_i$ satisfies $W^{\wt P_i} = \wt
  W_i$.
  Visually we can build the level $s$ of $\wt \sP$ by stacking together
  subsets of
  $[0,1]^{r[k]}$ that correspond to $\wt W_j$, ranged over all indices
  $j$ at
  level $s$. Then $P_{n,i}/\cQ_{n,s} = W_{n,i}/\cQ_{n,s} \to_{d_1} \wt
  W_{i,s} / \wt \cQ_{s} = \wt W_i / \wt \cQ_{s}=  \wt P_i / \wt \cQ_s$,
  so (c)
  is satisfied. Also (d) is satisfied since $W^{\wt P_i}_{\wt \cQ_s} =
  \wt W_{i,s} \to \wt W_i = W^{\wt P_i}$ pointwise almost everywhere as
  $s \to \infty$ (from our application of the Martingale Convergence
  Theorem).
\end{proof}

\begin{proof}[Proof of Theorem~\ref{thm:hypergraphon-converge-limit}]
  Let $\sP_n$ be the degree $(2,1,1,1,\dots)$ branching partition built
  from $W_n$
  as in Example~\ref{ex:W-branch}. Proposition~\ref{prop:part-converge}
  implies
  that there exists a branching partition $\wt \sP$ so that $\sP_n
  \dashto \wt\sP$
  along a subsequence, and hence $\sP_n \to \wt \sP$ along a subsequence by
  Lemma~\ref{lem:part-left}.
  Let $\wt P$ be the index $(1)$ element of $\wt \sP$. The associated
  hypergraphon $\wt W = W^{\wt P}$ is the desired limit of $W_n$. By
  applying
  \eqref{eq:left-conv-part} with $l = 1$ and $\alpha \equiv 1$, we see
  that
  $t(F,W_n) \to
  t(F,\wt W_n)$ along the subsequence.
\end{proof}

We conclude the paper with a conjecture that partitionable convergence is
equivalent to left-convergence, thereby proposing a converse to
Lemma~\ref{lem:part-left}.

\begin{conjecture} \label{conj:left-part}
  $\sP_n \to \wt\sP$ if and only if $\sP_n \dashto \wt\sP$.
\end{conjecture}

\section*{Acknowledgments}

The author would like to thank Jacob Fox, L\'aszl\'o Lov\'asz,
Jennifer Chayes and Christian Borgs for helpful conversations, and
also the anonymous referees for helpful comments and for pointing out
the connections between graph limits and exchangeable random arrays.

%\bibliographystyle{abbrv}
%\bibliography{ref_hyplim}

\end{document}